\documentclass[a4paper,11pt]{amsart}

\usepackage[utf8]{inputenc}
\usepackage{lmodern}
\usepackage{url}
\usepackage{amsmath, amsthm, amssymb, amscd, accents, bm}
\usepackage{mathtools}
\usepackage{mdwlist}
\usepackage{paralist}
\usepackage{graphicx}
\usepackage{datetime}
\usepackage{hyperref}
\usepackage{caption}
\usepackage[sort,nocompress]{cite}
\mathtoolsset{showonlyrefs}
\usepackage{units}
\usepackage{tikz}
\usepackage{colonequals}

\newtheorem{definition}{Definition}[section]
\newtheorem{theorem}[definition]{Theorem}
\newtheorem{lemma}[definition]{Lemma}
\newtheorem{corollary}[definition]{Corollary}
\newtheorem{proposition}[definition]{Proposition}

\newtheorem{remark}[definition]{Remark}

\def\N{{\mathbb N}}
\def\Z{{\mathbb Z}}
\def\R{{\mathbb R}}

\def\C{{\mathbb C}}

\newcommand{\lt}{{L^2(\R)}}

\newcommand{\im}{{\mathrm{Im} \,}}

\makeatletter
\newcommand\thankssymb[1]{\textsuperscript{\@fnsymbol{#1}}}

\begin{document}

\title[Multi-window STFT phase retrieval: lattice uniqueness]{Multi-window STFT phase retrieval:\\ lattice uniqueness}

\author[Philipp Grohs]{Philipp Grohs}
\address{Faculty of Mathematics, University of Vienna, Oskar-Morgenstern-Platz 1, 1090 Vienna, Austria}
\address{Research Network DataScience@UniVie, University of Vienna, Kolingasse 14-16, 1090 Vienna, Austria}
\address{Johann Radon Institute of Applied and Computational Mathematics, Austrian Academy of Sciences, Altenbergstrasse 69, 4040 Linz, Austria}
\email{philipp.grohs@univie.ac.at}

\author[Lukas Liehr]{Lukas Liehr}
\address{Faculty of Mathematics, University of Vienna, Oskar-Morgenstern-Platz 1, 1090 Vienna, Austria}
\email{lukas.liehr@univie.ac.at}

\author[Martin Rathmair]{Martin Rathmair}
\address{Institut de Math\'ematiques de Bordeaux, Universit\'e Bordeaux, UMR CNRS 5251, 351 Cours
de la Lib\'eration 33405, Talence, France}
\email{martin.rathmair@math.u-bordeaux.fr}


\begin{abstract}
Short-time Fourier transform (STFT) phase retrieval refers to the reconstruction of a function $f$ from its spectrogram, i.e., the magnitudes of its short-time Fourier transform $V_gf$ with window function $g$. While it is known that for appropriate windows, any function $f \in L^2(\mathbb{R})$ can be reconstructed from the full spectrogram $|V_g f(\mathbb{R}^2)|$, in practical scenarios, the reconstruction must be achieved from discrete samples, typically taken on a lattice. It turns out that the sampled problem becomes much more subtle: recent results have demonstrated that uniqueness via lattice-sampling is unachievable, irrespective of the choice of the window function or the lattice density. In the present paper, we initiate the study of multi-window STFT phase retrieval as a way to effectively bypass the discretization barriers encountered in the single-window case. By establishing a link between multi-window Gabor systems, sampling in Fock space, and phase retrieval for finite frames, we derive conditions under which square-integrable functions can be uniquely recovered from spectrogram samples on a lattice.
Specifically, we provide conditions on window functions $g_1, \dots, g_4 \in \lt$, such that every $f \in L^2(\mathbb{R})$ is determined up to a global phase from
$$
\left(|V_{g_1}f(A\mathbb{Z}^2)|, \, \dots, \, |V_{g_4}f(A\mathbb{Z}^2)| \right)
$$
whenever $A \in \mathrm{GL}_2(\mathbb{R})$ satisfies the density condition $|\det A|^{-1} \geq 4$. For real-valued functions, a density of $|\det A|^{-1} \geq 2$ is sufficient. Corresponding results for irregular sampling are also shown.
\end{abstract}

\maketitle

\section{Introduction}

The problem of recovering a function from its spectrogram, i.e., the absolute value of its short-time Fourier transform (STFT), forms a crucial step in several important application problems of current interest, which range from coherent diffraction imaging \cite{Pfeiffer2018} to quantum mechanics \cite{qm}. The resulting inverse problem is commonly known as the \emph{STFT phase retrieval problem}. The ubiquitous nature of the STFT phase retrieval problem and its relevance across numerous subjects lead to an intense investigation in recent years \cite{Alaifari2019,grohs2019stable,GrohsRathmair,grohsliehr2}. The problem is studied from a range of angles, including finite dimensional formulations \cite{Bojarovska2016,Strohmer,8866750}, group-theoretical settings \cite{fuhr2023phase,Bartusel2023,BARTUSELfuehr}, numerical analysis \cite{eldar,goetz,Iwen2022}, and frame-theoretical aspects \cite{alharbi1, daub, bandeira_savingphase, Bodmann2015,CONCA2015346,alaifariGrohs}.

Briefly speaking, the STFT phase retrieval problem asks to invert the map
$$
f \mapsto |V_gf(\Lambda)| \coloneqq \left ( |V_gf(z)| \right )_{z \in \Lambda}
$$
which sends a square-integrable function $f \in \lt$ to its spectrogram $|V_gf|$ sampled on $\Lambda \subseteq \R^2$. The map $V_gf$ denotes the short-time Fourier transform of $f$ with respect to the window-function $g \in \lt$ and is defined as
\begin{equation}\label{eq:STFT}
    V_gf(x,\omega) = \int_\R f(t)\overline{g(t-x)} e^{-2\pi i t\omega}\,dt, \quad x,\omega\in \R.
\end{equation}
Note that multiplying $f$ by a complex scalar of unit modulus does not change the corresponding spectrogram and therefore a reconstruction is only possible up to the ambiguity of a global phase factor. It is a classical result that mild assumptions on a window function $g$ imply that every $f \in \lt$ is determined (up to a global phase) by $|V_gf(\Lambda)|$, provided that $\Lambda = \R^2$ \cite[Theorem A.3]{GrohsRathmair}. 

In practical applications one only has access to samples of the spectrogram (in this case, the STFT phase retrieval problem is also termed \emph{phaseless sampling problem}). It is therefore crucial to determine to which extent uniqueness statements from $\Lambda=\R^2$ transfer to uniqueness statements from discrete sampling sets $\Lambda$, most importantly lattices, i.e. $\Lambda = A \Z^2$ for some invertible matrix $A \in \mathrm{GL}_2(\R)$. It has recently been shown in \cite{grohsLiehrJFAA,alaifari2020phase, nonuniqueness_theory} that the reconstruction from lattice samples is impossible:

\vspace{0.2cm}

\begin{quote}
\centering
    {\it there exists no window function and no lattice $\Lambda$ such that every $f \in\lt$ is determined up to a global phase by $|V_gf(\Lambda)|$.}
\end{quote}

\vspace{0.2cm}

Hence, sampling on lattices does not lead to a discretization of the STFT phase retrieval problem, no matter how dense the lattice is chosen.
This is in stark contrast to the setting where phase information is present, as classical results in time-frequency analysis demonstrate stable recovery of any $f \in \lt$ from $V_gf(A\Z^2)$, provided that $A\Z^2$ obeys a suitable density condition, and $g$ is suitably chosen, see \cite{Groechenig}.

Since applications demand the ability to reconstruct from samples, the question of how to adapt the sampling acquisition, so as to bypass the previously mentioned non-uniqueness property, is of importance.

In the present article, we propose an increase of the sampling redundancy beyond samples arising from a single window function.
Specifically, we propose a multi-window approach to the phaseless sampling problem in order to overcome the discretization barriers encountered in the single-window case. Multi-window Gabor systems appeared in the time-frequency analysis literature, for instance, in \cite{balan_multiwindow, groechenig_sharpresults,ZIBULSKI1997188}.
Our main result can be stated in a concise way as follows:

\vspace{0.1cm}

\begin{quote}
    \centering
    {\it there exist window functions $g_1,g_2,g_3,g_4$ such that for every $A\in \mathrm{GL}_2(\R)$ with  $|\det(A)|^{-1} \geq 4$, every $f \in \lt$ is determined up to a global phase by
$$
\left(
|V_{g_1}f(A\Z^2)|, \ldots, |V_{g_4}f(A\Z^2)| \right).
$$ }
\end{quote}

\vspace{0.1cm}

The window functions $g_1,g_2,g_3,g_4$ in the latter statement are of a rather simple form as they can be chosen as linear combinations of 
a Gaussian and the first Hermite function. Our results therefore show that spectrogram samples with respect to four window functions contain enough information to recover any square-integrable, complex-valued function. By contrast, previous results only showed that uniqueness via sampling on lattices is possible under severe restrictions of the function class to subspaces of $\lt$ \cite{grohsliehr1,grohs2022completeness,wellershoff2022injectivity}.

The theorems derived in the present exposition are designed in such a way that they give an easy-to-check condition on $g_1,\dots,g_4$ such that unique recovery via phaseless sampling is guaranteed. In addition, the statements are flexible to produce phaseless sampling results from non-uniform sampling sets, different from lattices (see Theorem \ref{theorem:lattice_result} in Section \ref{section:spectrogram_sampling}). 

Recall that there exists a one-to-one correspondence between the STFT with Gaussian window and functions in the Bargmann Fock space. Despite the fact that Fock spaces contain a rich theory on uniqueness, sampling and interpolation, this theory did not play any significant role so far in the study of the STFT phase retrieval problem. The proof techniques presented in the present exposition demonstrate, that the multi-window approach opens the door to apply sampling results in Fock spaces to the problem of phaseless sampling of the STFT, highlighting an underlying connection between the two subjects and representing a technique in passing from classical STFT sampling to phaseless STFT sampling. Importantly, the solution to a finite-dimensional phase retrieval problem serves as the essential link between ordinary sampling and phaseless sampling of the STFT.

\subsection{Contributions}\label{section:contribution}

We consider elements $f,h$ of a vector space over the complex field $\C$ equivalent, and write $f\sim h$, if there exists a constant $\tau \in \mathbb{T} \coloneqq \{ z \in \C : |z|=1\}$ such that $f=\tau h$. If $f \sim h$ then we also say that $f$ and $h$ agree up to a global phase. The next definition settles the notion of multi-window STFT phase retrieval.

\begin{definition}
Let $I$ be an index set, let $\{g_p\}_{p \in I}\subseteq L^2(\mathbb{R})$ be a family of window functions, 
let $\Lambda\subseteq \R^2$ be a set of sampling points, and let $\mathcal{C}\subseteq L^2(\mathbb{R})$ be a function class.
We say that $\left(\{g_p\}_{p\in I}, \Lambda \right)$ \emph{does phase retrieval} on $\mathcal{C}$ if it holds that 
$$
\left(f,h\in \mathcal{C}:~|V_{g_p} f(\Lambda)| = |V_{g_p} h(\Lambda)|,~ \forall p \in I\right) \quad \implies \quad f\sim h.
$$
\end{definition}
With a slight abuse of notation, we will say that a family of vectors $\mathcal{V}\subseteq \C^n$ does phase retrieval if it holds that 
$$
\left( z,w\in\C^n:~ |\langle z,\varphi\rangle| = |\langle w,\varphi \rangle|, ~\forall \varphi\in\mathcal{V}\right)
\quad
\implies
\quad
z\sim w.
$$
Having settled the above terminology, we can now turn to discuss the main results of the article. To that end, recall that the first two Hermite functions $h_0,h_1 \in \lt$ are given by
$$
h_0(t)=2^{1/4}e^{-\pi t^2}, \ \ \ h_1(t)=2^{5/4}\pi t e^{-\pi t^2}.
$$
We will employ window functions which arise as linear combinations of $h_0$ and $h_1$, and define 
\begin{equation}\label{def:gp}
    g_p:= \lambda h_0 + \mu h_1, \quad p=(\lambda,\mu)\in \C^2.
\end{equation}
We call $p \in \C^2$ the defining vector of $g_p$. Now consider a subset $\mathcal{P} \subseteq \C^2$ and regard $\mathcal{P}$ as an index set. This results in a family of window functions $\{ g_p \}_{p \in \mathcal{P}} \subseteq \lt$, parameterized by $\mathcal{P}$.

We are prepared to formulate statements that reveal a connection between phase retrieval in $\C^2$ and multi-window STFT phase retrieval. The first statement addresses the situation when the function space $\mathcal{C}$ is the entire space $\lt$.

\begin{theorem}\label{thm:mainlatticeversion}
Suppose that $\mathcal{P}\subseteq \C^2$ does phase retrieval.
If $A\in \mathrm{GL}_2(\R)$ satisfies $|\det(A)|^{-1} \ge 4$ then $\left(\{g_p\}_{p\in\mathcal{P}}, A\mathbb{Z}^2\right)$ does phase retrieval on $L^2(\R)$.
\end{theorem}

In the restricted setting where the function space is assumed to consist only of real-valued functions ($\mathcal{C} = L^2(\R,\R)$), the density of the sampling points can be reduced by half compared to the complex regime. Notice that just like in the case of complex-valued functions, the STFT phase retrieval problem lacks uniqueness from lattice samples with respect to a single window function, even when the function class is restricted to real-valued functions \cite[Section 3.3]{grohsLiehrJFAA}. The density reduction in the real-valued regime reads as follows.

\begin{theorem}\label{thm:latticeuniqueness_real}
Suppose that $\mathcal{P}\subseteq \C^2$ does phase retrieval.
If $A = \mathrm{diag}(\alpha,\beta) \in \mathrm{GL}_2(\R)$  satisfies  $|\det(A)|^{-1} \ge 2$ then $\left(\{g_p\}_{p\in\mathcal{P}}, (0,\frac{\beta}{4})^T + A\mathbb{Z}^2\right)$ does phase retrieval on $L^2(\R,\R)$.
\end{theorem}

The two theorems above include as an assumption that $\mathcal{P}$ forms a family which does phase retrieval in $\C^2$.
It is a notoriously hard problem to decide whether a given frame does phase retrieval in $\C^n$ or not \cite{bandeira_savingphase,Heinosaari2013,Edidin,Vinzant,Bodmann2015}. 
With regards to the problem of phase retrieval in $\C^2$, it is known that the so-called \emph{$4n-4$ conjecture} holds, that is, in dimension $n=2$ no family of three or less vectors does phase retrieval, while a generic family of four (or more) vectors does the job \cite[Theorem 10]{bandeira_savingphase}.
Recall that $m$ vectors $\varphi_1, \dots, \varphi_m \in \C^n$ are called generic if they belong to a certain non-empty Zariski open subset of $\C^{n \times m} \simeq (\R^{n \times m})^2$. 
In this article, we will  abolish the assumption on a frame to be generic and establish a precise characterization for a system of four vectors to do phase retrieval in $\C^2$. This characterization provides us straightaway with an easy-to-check geometric condition on how to choose four window functions for the STFT phase retrieval problem.

\begin{theorem}\label{thm:prinC2}
Let $\varphi_0,\varphi_1,\varphi_2,\varphi_3\in \C^2$. 
Moreover, define
\begin{equation*}
    \lambda_k: = \langle \varphi_k, \varphi_0\rangle ,
    \quad 
    \mu_k:= \langle\varphi_k, \bigl( \begin{smallmatrix}0 & -1\\ 1 & 0\end{smallmatrix}\bigr) \overline{\varphi_0} \rangle, 
    \quad k\in\{1,2,3\}.
\end{equation*}
The following statements are equivalent:
\begin{enumerate}[i)]
\item $\{\varphi_0, \varphi_1,\varphi_2,\varphi_3\}\subseteq \C^2$ does phase retrieval.
\item
It holds that $\mu_k\neq 0$, $k\in\{1,2,3\}$, and that the numbers 
$
\lambda_1/\mu_1, \lambda_2/\mu_2, \lambda_3/\mu_3
$
are not collinear.
\end{enumerate}
\end{theorem}

By combining Theorem \ref{thm:mainlatticeversion}, Theorem \ref{thm:latticeuniqueness_real}, and Theorem \ref{thm:prinC2}, we arrive at the following explicit result.

\begin{corollary}\label{cor:collinear_main}
Suppose that $\lambda_1,\lambda_2,\lambda_3\in\C$ are not collinear and define 
$$
\mathcal{P}:=\left\{
\begin{pmatrix}
1\\
0
\end{pmatrix},
\begin{pmatrix}
\lambda_1\\
1
\end{pmatrix},
\begin{pmatrix}
\lambda_2\\
1
\end{pmatrix},\begin{pmatrix}
\lambda_3\\
1
\end{pmatrix}
\right\}.
$$
If $A\in \mathrm{GL}_2(\R)$ satisfies $|\det(A)|^{-1}\ge 4$, then 
$(\{g_p\}_{p\in\mathcal{P}}, A\mathbb{Z}^2)$ does phase retrieval on $L^2(\R)$. Moreover, if $A = \mathrm{diag}(\alpha,\beta) \in \mathrm{GL}_2(\R)$  satisfies  $|\det(A)|^{-1} \ge 2$, then $\left(\{g_p\}_{p\in\mathcal{P}}, (0,\frac{\beta}{4})^T + A\mathbb{Z}^2\right)$ does phase retrieval on $L^2(\R,\R)$.
\end{corollary}

\begin{remark}
    We are primarily concerned with function spaces $\mathcal{C}$ that are subspaces of $\lt$. However, we mention the following generalization of Theorem \ref{thm:mainlatticeversion} (generalizations of the other statements can be obtain analogously): for a given Schwartz function $g \in \mathcal{S}(\R)$ and a tempered distribution $f \in \mathcal{S}'(\R)$, we define the STFT of $f$ with respect to $g$ by
$$
V_gf(x,\omega) = \langle f, \pi(x,\omega)g \rangle_{\mathcal{S}'(\R) \times \mathcal{S}(\R)},
$$
where $\pi(x,\omega)g(t) =  e^{2\pi i \omega t}g(t - x)$. Moreover, we define for a fixed non-zero Schwartz function $g$ the space of distributions with bounded STFTs via
    $$
    M^\infty \coloneqq \left \{ f \in \mathcal{S}'(\R) : \| V_gf \|_{L^\infty(\R^2)} < \infty \right \}.
    $$
    In fact, the latter definition is independent of the particular choice of the non-zero Schwartz function $g$.
    If one replaces in Theorem \ref{theorem:lattice_result} the density condition $|\det(A)|^{-1} \geq 4$ by a strict inequality, then one arrives at a phaseless sampling result in $M^\infty$: if $\mathcal{P} \subseteq \C^2$ does phase retrieval and if $A \in \mathrm{GL}_2(\R)$ satisfies $|\det(A)|^{-1} > 4$, then $\left(\{g_p\}_{p\in\mathcal{P}}, A\mathbb{Z}^2\right)$ does phase retrieval on $M^\infty$. That is, for every $f,h \in M^\infty$ one has the implication
    $$
    |V_{g_p} f(\Lambda)| = |V_{g_p} h(\Lambda)|,~ \forall p \in \mathcal{P} \quad \implies \quad f\sim h.
    $$
\end{remark}

\begin{figure}
\centering
\hspace*{-1.6cm}
  \includegraphics[width=15.5cm]{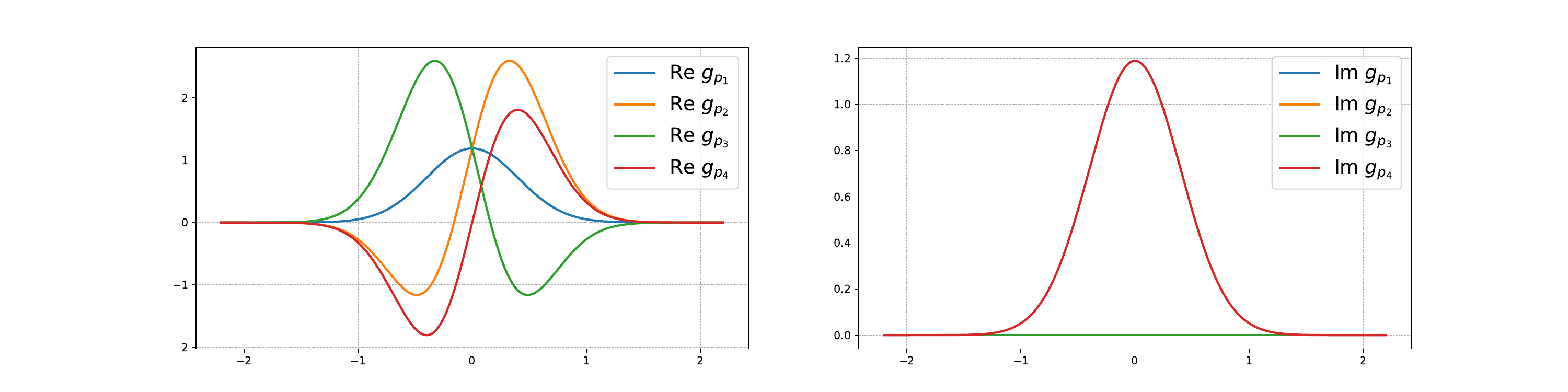}
\caption{{Plots of the real and imaginary part of four window functions $g_{p_1}, \dots, g_{p_4}$ subjected to the point configuration $\mathcal{P} = \{ p_1, \dots, p_4 \} \subseteq \C^2$ given by $\mathcal{P}=\left \{ \begin{pmatrix} 1\\0 \end{pmatrix}, \begin{pmatrix} 1\\1 \end{pmatrix}, \begin{pmatrix} -1\\1 \end{pmatrix}, \begin{pmatrix} i\\1 \end{pmatrix} \right \}.$
According to Corollary \ref{cor:collinear_main}, $\left(\{g_p\}_{p\in\mathcal{P}}, A\mathbb{Z}^2\right)$ does phase retrieval on $L^2(\R)$, provided that $|\det(A)|^{-1} \ge 4$.}
}
\label{figure:plot_windows}
\end{figure}

\subsection{Terminology and notation}\label{sec:term_not}

Throughout the exposition, we will casually identify $\R^2$ with the complex plane $\C$ by virtue of the map $(x,y)^T \mapsto x+iy$. For a subset $S \subseteq \C$ we denote by $\overline{S}$ the set consisting of the conjugate elements of $S$, i.e., $\overline{S} \coloneqq \{ s \in \C : \overline{s} \in S \}$. Similarly, if $S \subseteq \R^2$ then $\overline{S}$ stands for the set $\overline{S} \coloneqq \{ (x,y)^T \in \R^2 : (x,-y)^T \in S \}$.

For $\Omega$ a set and $\mathcal{A}\subseteq \C^\Omega$ a class of functions, we say that $Z \subseteq \Omega$ is a set of uniqueness for $\mathcal{A}$ if 
$$
\left(f,g \in \mathcal{A}: f(z) = g(z), \ \forall z \in Z \right) \quad \implies \quad f=g.
$$

A subset $\Lambda \subseteq \R^2 \simeq \C$ is called a shifted lattice if there exists a vector $v \in \R^2$ and an invertible matrix $A \in \mathrm{GL}_2(\R)$ such that
$$
\Lambda = v + A\Z^2 = \{ v+Ax : x \in \Z^2 \}.
$$
The matrix $A$ is called the generating matrix of $\Lambda$. A shifted lattice $\Lambda$ is called separable if it is generated by a diagonal matrix. For a matrix $A \in \C^{2 \times 2}$ we denote by $A^*$ the conjugate transpose of $A$, i.e. $A^* = (\overline{A})^T$.

As usual, we denote by $\lt$ the Lebesgue space of all measurable, complex-valued and square-integrable functions $f : \R \to \C$. The subspace of $\lt$ which consists of all real-valued functions in $\lt$ is denoted by $L^2(\R,\R)$.

The short-time Fourier transform $V_gf$ of a function $f \in \lt$ with respect to a window function $g \in \lt$ is defined as in equation \eqref{eq:STFT}. The map $V_gf$ is uniformly continuous and satisfies the relation $\| V_gf \|_{L^2(\R^2)} = \| f \|_{L^2(\R)} \| g \|_{L^2(\R)}$ \cite[Corollary 3.2.2]{Groechenig}. Hence, $f \mapsto V_gf$ constitutes an isometry from $L^2(\R)$ into $L^2(\R^2)$, provided that $\| g \|_{L^2(\R)} = 1$.

Finally, the vector space $\C^n$, $n \in \N$, is equipped with the scalar product $\langle v,w \rangle = \sum_{j=1}^n v_j \overline{w_j}$ where $v=(v_1,\dots,v_n)^T \in \C^n$ and $w=(w_1,\dots,w_n)^T \in \C^n$.

\subsection{Outline}

In Section \ref{section:fock_spaces}, we collect preliminary results concerning properties of Fock spaces with a focus on sets of uniqueness and stable sampling, which will be needed throughout the paper. Section \ref{section:pr_C2} is devoted to prooving of a characterization of vectors that do phase retrieval in $\C^2$. Based on the groundwork laid in Section \ref{section:fock_spaces} and Section \ref{section:pr_C2}, we proceed to derive and utilize a connection between phase retrieval in a finite frame setting, sets of uniqueness in Fock spaces, as well as phaseless sampling of the STFT. In particular we prove the phaseless sampling results stated in Section \ref{section:contribution}, including several extensions to irregular sampling.

\section{Preliminaries on Fock spaces}\label{section:fock_spaces}

This section is devoted to recalling and collecting a couple of facts about Fock spaces, that are used throughout the remainder of the article. For an exposition on Fock spaces we refer to \cite{zhu:fock}.

\subsection{Basic properties and relations to STFT}

The space $L^2_\alpha(\C)$ consists of all Lebesgue measurable functions $F : \C \to \C$  for which
$$
\| F \|_{\alpha} \coloneqq \left ( \frac{\alpha}{\pi} \int_\C |F(z)|^2 e^{-\alpha |z|^2} \, dA(z) \right )^{\frac{1}{2}} < \infty,
$$
where $dA(z)$ denotes the Euclidean area measure on $\C$. If $\mathcal{O}(\C)$ is the collection of all entire function on $\C$ then the Fock space (or Bargmann-Fock space) is defined as the intersection
$$
\mathcal{F}^2_\alpha(\C) \coloneqq L_\alpha^2(\C) \cap \mathcal{O}(\C).
$$
The pointwise estimate
\begin{equation}\label{eq:fockptwise}
|F(z)| \le e^{\frac{\alpha}{2} |z|^2} \|F\|_{\alpha}
\end{equation}
which holds for every $F \in \mathcal{F}^2_\alpha(\C)$ and every $z \in \C$, renders $\mathcal{F}^2_\alpha(\C)$ into a reproducing kernel Hilbert space (RKHS) with inner product 
$$
\langle F,G \rangle_\alpha = \frac{\alpha}{\pi} \int_\C F(z) \overline{G(z)} e^{-\alpha |z|^2} \, dA(z).
$$
In addition, the RKHS structure implies the following statement.

\begin{lemma}\label{lem:fockproduct}
Let $\alpha,\beta>0$ and suppose that $F\in \mathcal{F}^2_\alpha(\C)$ and $G\in \mathcal{F}_\beta^2(\C)$.
Then it holds that $FG\in \mathcal{F}^2_{\alpha+\beta}(\C)$.
\end{lemma}
\begin{proof}
Using the pointwise estimate given in equation \eqref{eq:fockptwise} shows that
\begin{align*}
    \|FG\|_{\alpha+\beta}^2 &= \frac{\alpha+\beta}\pi \int_\C |F(z)G(z)|^2 e^{-(\alpha+\beta)|z|^2} \,dA(z) \\
    &\le \frac{\alpha+\beta}\pi \|F\|_{\alpha}^2 \int_\C |G(z)|^2 e^{-\beta|z|^2}\,dA(z)\\
    &=  \frac{\alpha+\beta}\beta \|F\|_{\alpha}^2 \|G\|_{\beta}^2 < \infty,
\end{align*}
as desired.
\end{proof}

Given a function $F \in \mathcal{F}^2_\alpha(\C)$, we define the Bargmann-shift of $F$ by $a \in \C$ as
\begin{equation}\label{eq:shift_operator}
    T_aF(z) \coloneqq e^{\alpha z \overline{a} - \frac{\alpha}{2}|a|^2} F(z-a).
\end{equation}
According to \cite[Proposition 2.38]{zhu:fock}, the operator $T_a$ is a unitary operator on $\mathcal{F}^2_\alpha(\C)$ for every $a \in \C$.
If $f \in \lt$, then the Bargmann transform $Bf$ of $f$, defined by 
$$
Bf(z)\coloneqq 2^{1/4} \int_\R f(t) e^{2\pi t z -\pi t^2-\frac{\pi}2 z^2}\,dt,
$$
satisfies $Bf \in \mathcal{F}_\pi^2(\C)$. In fact, the Bargmann transform $B$ is a unitary map from $L^2(\R)$ onto $\mathcal{F}_\pi^2(\C)$ \cite[Theorem 3.4.3]{Groechenig}.
There is an intimate relationship between the STFT with Hermite windows and the Bargmann transform. 
For our purposes two identities will be of significance: with  
\begin{equation}\label{def:eta}
    \eta(z):=\exp \left(-\pi ixy+\frac\pi2 |z|^2\right), \quad z=x+iy\in\C,
\end{equation}
it holds for all $f\in L^2(\R)$ and for all $z=x+iy$ that 
\begin{align}
    V_{h_0} f(x,-y) \cdot \eta(z)&= Bf(z),\label{eq:Vh0bargmann} \\
    V_{h_1} f(x,-y) \cdot \eta(z)&= (Bf)'(z) -\pi \bar{z} Bf(z). \label{eq:Vh1bargmann}
\end{align}

Note that, up to a non-zero weighting factor, the right-hand side of equation \eqref{eq:Vh1bargmann} defines the so-called translation-invariant derivative of $Bf$ \cite[Remark 1]{brekke1993density}. Derivatives of this type play a central role in the theory of sampling with derivatives in Fock space, as discussed in a classical paper by Brekke and Seip \cite{brekke1993density}.

\subsection{Sets of uniqueness and sampling}

Recall that a set $Z \subseteq \C$ is a set of uniqueness for $\mathcal{F}_\alpha^2(\C)$, if every function in $\mathcal{F}_\alpha^2(\C)$ that vanishes on $Z$ must vanish identically (since $\mathcal{F}_\alpha^2(\C)$ is a vector space, this definition of a set of uniqueness is consistent with the one introduced in Section \ref{sec:term_not}). If $Z= \{ z_n : n \in \N \}$ is a sequence of distinct points in $\C$, then $Z$ is said to be a set of stable sampling for $\mathcal{F}_\alpha^2(\C)$ if there exists a constant $C>0$ such that
$$
C^{-1} \| F \|_{\alpha}^2 \leq \sum_{n=1}^\infty |F(z_n)|^2 e^{-\alpha |z_n|^2} \leq C \| F \|_\alpha^2
$$
for all $F \in \mathcal{F}_\alpha^2(\C)$. Clearly, every set of stable sampling is a set of uniqueness. Further, we say that $Z \subseteq \C$ is separated (or: uniformly discrete) if
$$
\inf_{\substack{z,z' \in Z \\ z \neq z'}} |z - z'| > 0.
$$
Denoting by $B_r(w) \coloneqq \{ z \in \C : |z-w|<r \}$ the open ball of radius $r>0$ around $w \in \C$ and by $\#(\Omega)$ the number of elements in $\Omega \subseteq \C$, then the lower Beurling density of $Z \subseteq \C$ is defined as
$$
D^-(Z) \coloneqq \liminf_{r \to \infty} \left ( \inf_{w \in \C} \frac{\#(Z \cap B_r(w))}{\pi r^2} \right ).
$$
For instance, if $Z = A\Z^2$, $A \in \mathrm{GL}_2(\R)$, is a lattice then $D^-(Z) = |\det(A)|^{-1}$. Among lattices, sets of uniqueness are characterized in terms of their density.

\begin{theorem}[Perelomov \cite{perelomov71}]\label{thm:perelomovuniq}
A lattice $A\Z^2 \subseteq \C$, $A \in \mathrm{GL}_2(\R)$, is a set of uniqueness for $\mathcal{F}_\alpha^2(\C)$ if and only if $|\det(A)|^{-1} \geq \alpha/\pi$.
\end{theorem}

For general separated sets one has the following characterization of sets of stable sampling.

\begin{theorem}[Lyubarskii \cite{lyubarskiiframes}, Seip and Wallstén \cite{Seip1992}]\label{thm:seipsampling}
If $Z \subseteq \C$ is separated then $Z$ is a set of stable sampling for $\mathcal{F}_\alpha^2(\C)$ if and only if $D^-(Z) > \alpha/\pi$.
\end{theorem}

With the aid of the Bargmann-shift, we can easily verify that the property of a set $Z \subseteq \C$ being a set of uniqueness for $\mathcal{F}_\alpha^2(\C)$ is invariant under translations. For if $Z$ is a set of uniqueness and if $v \in \C$ then a function $F \in \mathcal{F}_\alpha^2(\C)$ vanishes on $v+Z = \{ v+z : z \in Z \}$ if and only if $T_{-v}F \in \mathcal{F}_\alpha^2(\C)$ vanishes on $Z$. Hence, $T_{-v}F \equiv 0$ and the fact that $T_{-v}F$ equals to $F(\cdot + v)$ up to a non-zero weighting factor forces $F$ to vanish identically.

Notice that in contrast to sets of stable sampling, the general characterization of sets of uniqueness in $\mathcal{F}_\alpha^2(\C)$ constitutes a difficult problem, and sets of uniqueness can have an unexpected structure. For instance, Ascensi, Lyubarskii and Seip showed that if
$$
\mathcal{R} \coloneqq \{ (1,0),(-1,0) \} \cup \{ (\pm \sqrt{2n},0) : n \in \N \} \cup \{ (0,\pm \sqrt{2n}) : n \in \N \}
$$
then $\mathcal{R}$ is a set of uniqueness for $\mathcal{F}_\pi^2(\C)$. The set $\mathcal{R}$ has the property that there exist arbitrarily large disks containing no points from $\mathcal{R}$ \cite{ASCENSI2009277}. In particular, it holds that $D^-(\mathcal{R}) = 0$. Finally, we mention a necessary density condition obtained by Belov, Borichev and Kuznetsov who proved in \cite{BELOV2020438} that if $Z \subseteq \C$ is a set of uniqueness for $\mathcal{F}_\pi^2(\C)$ then
$$
\limsup_{r \to \infty} \frac{\#(Z \cap B_r(0))}{\pi r^2} \geq \frac{1}{3\pi}.
$$

\section{Phase retrieval in $\C^2$}\label{section:pr_C2}

The objective of the present section is to establish a full characterization of phase retrieval in $\C^2$. Precisely, we prove the following statement.

\begin{theorem}\label{thm:prinC2_2}
Let $\varphi_0,\varphi_1,\varphi_2,\varphi_3\in \C^2$. 
Moreover, define
\begin{equation*}
    \lambda_k: = \langle \varphi_k, \varphi_0\rangle ,
    \quad 
    \mu_k:= \langle\varphi_k, \bigl( \begin{smallmatrix}0 & -1\\ 1 & 0\end{smallmatrix}\bigr) \overline{\varphi_0} \rangle, 
    \quad k\in\{1,2,3\}.
\end{equation*}
The following statements are equivalent:
\begin{enumerate}[i)]
\item $\{\varphi_0, \varphi_1,\varphi_2,\varphi_3\}\subseteq \C^2$ does phase retrieval.
\item
It holds that $\mu_k\neq 0$, $k\in\{1,2,3\}$ and that the numbers 
$
\lambda_1/\mu_1, \lambda_2/\mu_2, \lambda_3/\mu_3
$
are not collinear.
\end{enumerate}
\end{theorem}

\subsection{Auxiliary results}

Before we turn towards the proof of Theorem \ref{thm:prinC2_2}, we collect a few important facts.
Firstly, we require the fact that the property of doing phase retrieval is invariant under invertible linear transformations.
\begin{lemma}\label{lem:prframetrafo}
Let $\Phi=\{\varphi_k\}_{k\in I}\subseteq \C^n$ and $A \in \mathrm{GL}_n(\C)$. Then $\Phi$ does phase retrieval in $\C^n$ if and only if $\tilde{\Phi}:=\{A\varphi_k\}_{k\in I}$ does phase retrieval in $\C^n$.
\end{lemma}
\begin{proof}
Clearly -- since one may replace $\Phi$ by $\tilde{\Phi}$ and $A$ by $A^{-1}$, respectively -- it suffices to show one implication.

Let us assume that $\Phi$ does phase retrieval and suppose that 
$z,z'\in \C^n$ are such that 
\begin{equation}\label{eq:assptPhidoespr}
|\langle z,A\varphi_k\rangle| = |\langle z',A\varphi_k\rangle|,~\quad \forall k\in I.
\end{equation}
We need to show that $z\sim z'$.
It follows from \eqref{eq:assptPhidoespr} that 
$$
|\langle A^* z, \varphi_k\rangle| = |\langle A^* z',\varphi_k\rangle|, \quad \forall k\in I.
$$
Since $\Phi$ does phase retrieval this implies that $A^*z \sim A^* z'$. 
Since $A^*$ is invertible we get that $z \sim z'$, and are done.
\end{proof}

Furthermore, the following geometric condition which guarantees that a point in the plane is uniquely determined by its distance from three fixed points, will play a central role.
\begin{lemma}\label{lemma:collinear}
Let $z,w,a_1,a_2,a_3 \in \C$ such that
\begin{equation}\label{eq:threedistances}
|z-a_j| = |w-a_j|,\quad j\in\{1,2,3\}.
\end{equation}
If $a_1,a_2,a_3$ are not collinear then $z=w$.
Conversely, if $a_1,a_2,a_3$ are collinear there exist $z\neq w$ such that \eqref{eq:threedistances} holds.
\end{lemma}
\begin{proof}
Suppose by contradiction that $z \neq w$. Then there exists a unique line segment $L$ of finite length connecting $z$ and $w$. The assumption that $|z-a_j| = |w-a_j|$ for all $j \in \{ 1,2,3 \}$ implies that $a_1,a_2,a_3$ lie on the unique perpendicular bisector of the line segment $L$. This contradicts the assumption that $a_1,a_2,a_3$ are not collinear.

For the second assertion let $J$ denote a line which contains the three points $a_1,a_2,a_3$. Since the points are collinear such a line exists.
Let $z$ be any point in $\C\setminus J$, and let $w$ be the reflection of $z$ across $J$. Then $z\neq w$ satisfy \eqref{eq:threedistances}. 
\end{proof}
We are now well-equipped to deal with the following simple configuration.
\begin{proposition}\label{prop:prinC2particular}
Let $\beta_1,\beta_2,\beta_3\in \C$. The family 
\begin{equation*}
    \left\{
    \begin{pmatrix}
    1\\
    0
    \end{pmatrix},
    \begin{pmatrix}
    \beta_1\\
    1
    \end{pmatrix},
    \begin{pmatrix}
    \beta_2\\
    1
    \end{pmatrix},
    \begin{pmatrix}
    \beta_3\\
    1
    \end{pmatrix}
    \right\} \subseteq\C^2
\end{equation*}
does phase retrieval if and only if $\beta_1,\beta_2,\beta_3$ are not collinear. 
\end{proposition}
\begin{proof}
We set $\varphi_0:=(1,0)^T$, and $\varphi_k:=(\beta_k,1)^T$, $k\in\{1,2,3\}$. We proceed in two steps and show
\begin{enumerate}[a)]
    \item If $\beta_1, \beta_2,\beta_3$ are not collinear, then $\{\varphi_k\}_{k=0}^3$ does phase retrieval. 
    \item If $\beta_1,\beta_2,\beta_3$ are collinear, then $\{\varphi_k\}_{k=0}^3$ does not do phase retrieval. 
\end{enumerate}
\textbf{Step a)}
Suppose that $\beta_1,\beta_2,\beta_3$ are not collinear, and let $z,w\in \C^2$ with $z=(z_1,z_2)^T$ and $w=(w_1,w_2)^T$ be such that 
\begin{equation}\label{eq:pridentities}
    |\langle z,\varphi_k \rangle| = |\langle w,\varphi_k\rangle|, \quad k\in \{0,1,2,3\}.
\end{equation}
We need to prove that $z\sim w$.

To that end, we distinguish two cases depending on $z_1=0$ or $z_1\neq 0$.
If $z_1=0$ it follows that 
$$
|w_1| = |\langle w,\varphi_0\rangle| = |\langle z,\varphi_0\rangle| = 0,
$$
and 
$$
|w_2|=  |\langle w,\varphi_1\rangle| = |\langle z,\varphi_1\rangle| = |z_2|.
$$
Thus, we have that $z\sim w$ in this case.
If $z_1\neq 0$ we get that 
\begin{equation}
|w_1|^2 = |\langle w,\varphi_0\rangle|^2 = |\langle z,\varphi_0,\rangle|^2 = |z_1|^2\neq 0.
\end{equation}
Furthermore, by assumption we have that 
\begin{equation*}
    |\bar{\beta_k}z_1+z_2| = |\bar{\beta_k}w_1+w_2|, \quad k\in\{1,2,3\}.
\end{equation*}
Division by $|z_1|=|w_1|\neq 0$ implies that 
\begin{equation*}
    \left|\frac{z_2}{z_1}+\overline{\beta_k} \right| =     \left|\frac{w_2}{w_1}+\overline{\beta_k} \right|, \quad k\in\{1,2,3\}.
\end{equation*}
Note that $-\bar{\beta_1},-\bar{\beta_2},-\bar{\beta_3}$ are not collinear since $\beta_1,\beta_2,\beta_3$ are not collinear. Thus, applying Lemma \ref{lemma:collinear} yields that $z_2/z_1=w_2/w_1$, which is equivalent to 
$$
z_2 \frac{w_1}{z_1} = w_2.
$$
Hence with $\tau:=w_1/z_1\in \mathbb{T}$ we have indeed that $\tau z = w$. In particular, this implies that $z\sim w$.

\textbf{Step b)} Suppose now that $\beta_1,\beta_2,\beta_3$ are collinear.
We need to find a pair of vectors $z,w\in \C^2$ which are not equivalent while satisfying \eqref{eq:pridentities}.
According to Lemma \ref{lemma:collinear} there exist complex numbers $p\neq q$ such that $|p-\beta_k|=|q-\beta_k|$ for every $k\in\{1,2,3\}$.
Since $p\neq q$ we have that 
$$
z=\begin{pmatrix}
    -1\\
    \bar{p}
    \end{pmatrix},
    \quad 
w= \begin{pmatrix}
    -1\\
    \bar{q}
    \end{pmatrix}
$$
are not equivalent. Furthermore, it holds that $|\langle z,\varphi_0 \rangle|=1=|\langle w,\varphi_0\rangle|$, and for every $k\in\{1,2,3\}$ that
$$
|\langle z,\varphi_k \rangle| = |p-\beta_k| = |q-\beta_k| = |\langle w,\varphi_k\rangle|,
$$
which settles the proof.
\end{proof}
\subsection{Proof of Theorem \ref{thm:prinC2_2}}
The idea of the proof is to reduce the general situation to the rather particular configuration in Proposition \ref{prop:prinC2particular}.

First observe that if $\varphi_0=0$, the family $\{\varphi_k\}_{k=0}^3$ cannot do phase retrieval. Indeed, assume it does phase retrieval, then also $\{\varphi_k\}_{k=1}^3$ does phase retrieval. This leads to a contradiction as it is know that a family of vectors in $\C^2$ which does phase retrieval consists of at least four elements \cite[Theorem 10]{bandeira_savingphase}.

We may therefore assume that $\varphi_0=(p,q)^T \neq 0$ and define 
$$
A:=
\begin{pmatrix}
\bar{p} & \bar{q}\\
-q & p
\end{pmatrix} \in \mathrm{GL}_2(\C).
$$
We have that 
\begin{equation*}
    A\varphi_0 = \begin{pmatrix}
    |p|^2+|q|^2\\
    0
    \end{pmatrix},
\end{equation*}
and that for every $k\in\{1,2,3\}$
\begin{equation*}
    A\varphi_k = 
    \begin{pmatrix}
    \langle \varphi_k, \varphi_0\rangle\\
    \langle \varphi_k, \bigl( \begin{smallmatrix}0 & -1\\ 1 & 0\end{smallmatrix}\bigr) \overline{\varphi_0} \rangle
    \end{pmatrix}
    = 
    \begin{pmatrix}
    \lambda_k\\
    \mu_k
    \end{pmatrix}.
\end{equation*}
Applying Lemma \ref{lem:prframetrafo} yields that $\{\varphi_k\}_{k=0}^3$ does phase retrieval if and only if 
$$
\tilde{\Phi}:=
\left\{
\begin{pmatrix}
|p|^2+|q|^2\\
0
\end{pmatrix},
\begin{pmatrix}
\lambda_1\\
\mu_1
\end{pmatrix},
\begin{pmatrix}
\lambda_2\\
\mu_2
\end{pmatrix},
\begin{pmatrix}
\lambda_2\\
\mu_2
\end{pmatrix}
\right\}
$$
does phase retrieval.

It remains to show that $\tilde{\Phi}$ does phase retrieval if and only if condition ii) is fulfilled.
We begin with the first implication and assume that $\tilde{\Phi}$ does phase retrieval. It follows that $\mu_k\neq 0$ for $k\in\{1,2,3\}$; indeed if we had that $\mu_k=0$ for some $k$ this would imply that there exists a pair of linearly dependent vectors in $\tilde{\Phi}$. Then we could reduce $\tilde{\Phi}$ to a family of three vectors which does phase retrieval in $\C^2$, and, once more, end up with a contradiction.
Since the property of doing phase retrieval is invariant under multiplication by a non-zero scalar of the individual vectors we get that also 
$$
\Psi:=
\left\{
\begin{pmatrix}
1\\
0
\end{pmatrix},
\begin{pmatrix}
\lambda_1/\mu_1\\
1
\end{pmatrix},
\begin{pmatrix}
\lambda_2/\mu_2\\
1
\end{pmatrix},
\begin{pmatrix}
\lambda_3/\mu_3\\
1
\end{pmatrix}
\right\}
$$
does phase retrieval. 
Invoking Proposition \ref{prop:prinC2particular} yields then that the points
$$\lambda_1/\mu_1, \lambda_2/\mu_2, \lambda_3/\mu_3$$
are not collinear.

For the second implication, assume ii). According to Proposition \ref{prop:prinC2particular} we have that $\Psi$ does phase retrieval. By multiplying each of the individual vectors with appropriate non-zero scalars we obtain that $\tilde{\Phi}$ indeed does phase retrieval, and conclude the proof.

\section{Spectrogram sampling}\label{section:spectrogram_sampling}

This section merges the results presented in Section \ref{section:fock_spaces} and Section \ref{section:pr_C2} and establishes a new approach for deriving phaseless sampling results in the context of the STFT phase retrieval problem. We begin by considering an abstract result on phaseless sampling for general function classes $\mathcal{C} \subseteq \lt$. This is followed by an analysis of the specific case where the function space is the entire space $\lt$, and concludes with an investigation of the function class $\mathcal{C} = L^2(\R,\R)$.

\subsection{Phaseless sampling}\label{subsection:phaseless_sampling}

We start by establishing an abstract phase retrieval statement in order to highlight which ingredients are essential for our results.

\begin{proposition}\label{proposition:abstract_uniqueness_result}
Suppose that $\mathcal{P}\subseteq \C^2$ does phase retrieval.
Moreover, let $\mathcal{C}\subseteq L^2(\R)$ be a function class and let $\Lambda\subseteq\R^2$ be such that $\bar{\Lambda}$
is a set of uniqueness for
\begin{equation}\label{def:AofC}
\mathcal{A}(\mathcal{C}):=\{Bf:~f\in \mathcal{C}\} \cup \{Bf Bh((Bf)'Bh-(Bh)'Bf):~f,h\in\mathcal{C}\}.
\end{equation}
Then it holds that $(\{g_p\}_{p\in\mathcal{P}}, \Lambda)$ does phase retrieval on $\mathcal{C}.$
\end{proposition}

\begin{proof}
Suppose $f,h\in \mathcal{C}$ are such that for all $z\in\Lambda$ it holds that 
\begin{equation}\label{eq:sampids}
|V_{g_p}f(z)| = |V_{g_p} h(z)|, \quad p\in \mathcal{P}.
\end{equation}
We need to show that $f\sim h$. 
The proof consists of two principal steps:
First we consider a phase retrieval problem of a local flavour and show that for each point $z\in \Lambda$ it holds that 
\begin{equation}\label{eq:equivsamples}
\begin{pmatrix}
    V_{h_0}f(z)\\V_{h_1}f(z)
\end{pmatrix}
\sim
\begin{pmatrix}
    V_{h_0}h(z)\\V_{h_1}h(z)
\end{pmatrix}.
\end{equation}
The second step consists of combining the local information obtained in step one by making use of the assumed properties of the set $\Lambda$ in order to conclude that $f\sim h$.

\textbf{Step a)}
For any $z\in \Lambda$, $f\in L^2(\mathbb{R})$ and $p=(\lambda,\mu)\in\C^2$ we have that
$$
|V_{g_p}f(z)| = \left|\bar{\lambda}V_{h_0}f(z)+\bar{\mu}V_{h_1}f(z)\right| = \left|\Big\langle
\begin{pmatrix}
V_{h_0} f(z)\\
V_{h_1} f(z)
\end{pmatrix},
p
\Big\rangle \right|.
$$
Since $\mathcal{P}$ does phase retrieval  it follows from
\eqref{eq:sampids} that the equivalence
\eqref{eq:equivsamples}  holds for all $z\in \Lambda$.

\textbf{Step b)}
Let $F,H$ denote the Bargmann transforms of $f$ and $h$, respectively.
Using the identities \eqref{eq:Vh0bargmann} and \eqref{eq:Vh1bargmann} it is easy to see that \eqref{eq:equivsamples} for all $z\in \Lambda$ implies that
\begin{equation}\label{eq:eqivvectorsbargmann}
    \begin{pmatrix}
    F(z)\\
    F'(z)
    \end{pmatrix}
    \sim
    \begin{pmatrix}
    H(z)\\
    H'(z)
    \end{pmatrix}, \quad \forall z\in \bar{\Lambda}
\end{equation}
Thus, we have that 
$$
|F|^2|_{\bar{\Lambda}} = |H|^2|_{\bar{\Lambda}}
\quad \text{and}\quad 
(\overline{F}F')|_{\bar{\Lambda}} = (\overline{H}H')|_{\bar{\Lambda}}.
$$
The pivotal idea consist of combining these two identities in a shrewd manner, namely to make the observation that 
$$
|F|^2 \cdot \overline{H}H' - |H|^2 \cdot \overline{F}F' = \overline{FH} \left(FH'-F'H \right)
$$
vanishes on $\bar{\Lambda}$. Thus, also 
$G:=FH(FH'-F'H)$ vanishes on $\bar{\Lambda}$. Since $G\in \mathcal{A}(\mathcal{C})$ the assumption on $\Lambda$ implies that $G$ vanishes everywhere on $\C$.
Thus, since the product $G$ consists of  factors $F, H, FH'-F'H$ which are entire functions, at least one of them must be the zero function.

If $F=0$, it follows from  \eqref{eq:eqivvectorsbargmann} that $H$ vanishes on $\bar{\Lambda}$. Since $H\in\mathcal{A}(\mathcal{C})$ this implies that also $H$ vanishes identically. Therefore, we have that $f=h=0$ and in particular $f\sim h$.
In the same way one argues that $f\sim h$ in the case where $H$ vanishes identically.

To finish the proof it remains to consider the case where neither $F$ nor $H$ vanish identically, and where $FH'-F'H=0$. 
Let $E\subseteq \C$ be a domain where $F$ has no zeros, then we have for all $z\in E$ that 
$$
(H/F)'(z) = \frac{(FH'-F'H)(z)}{F^2(z)} = 0.
$$
Hence $H=cF$ on $E$ for suitable constant $c\in\C$. By analyticity we get that $H=cF$ on $\C$.

Furthermore, since $|H|=|F|$ on $\bar{\Lambda}$ we have for all $z\in \bar{\Lambda}$ that 
$$
|F(z)|=|H(z)|=|cF(z)|,
$$
which implies that $(1-|c|)F$ vanishes on $\bar{\Lambda}$.
By assumption, $F\in\mathcal{A}(\mathcal{C})\setminus\{0\}$, and therefore cannot vanish on all of $\bar{\Lambda}$. From this we conclude that $|c|=1$.
Since the Bargmann transform is injective it follows from 
$$
B(h-cf) = H-cF =0
$$
that indeed $h\sim f$.
\end{proof}

The role of the set $\Lambda$ in Proposition \ref{proposition:abstract_uniqueness_result} is still somewhat concealed at this stage. However, what is noticeable is that $\mathcal{A}(\mathcal{C})$ consists of entire functions only. We have thus reduced the problem of finding a configuration which does phase retrieval to the problem of identifying sets of uniqueness for sub-classes of $\mathcal{O}(\C)$.
In fact, we can say even more about $\mathcal{A}(\mathcal{C})$. 
\begin{lemma}\label{lem:admfock}
For all $\mathcal{C}\subseteq L^2(\R)$ it holds that $\mathcal{A}(\mathcal{C})\subseteq \mathcal{F}_{4\pi}^2(\C)$.
\end{lemma}
\begin{proof}
First we note that for all $f\in L^2(\R)$ it holds that $Bf\in \mathcal{F}_\pi^2(\C) \subseteq \mathcal{F}_{4\pi}^2(\C)$.
It remains to show that for all $f,h\in L^2(\R)$ it holds that 
\begin{equation}\label{eq:derivbargmannfock}
(Bf)'Bh- Bf (Bh)' \in \mathcal{F}_{2\pi}^2(\C).
\end{equation}
Indeed, from \eqref{eq:derivbargmannfock} it then follows by applying Lemma \ref{lem:fockproduct} twice that 
$$Bf Bh [(Bf)'Bh-Bf(Bh)']\in \mathcal{F}_{4\pi}^2(\C),$$ 
and we are done.
In order to prove \eqref{eq:derivbargmannfock} we use identity \eqref{eq:Vh1bargmann} to express the derivatives of the Bargmann transforms according to 
\begin{align*}
    (Bf)'(z) &= V_{h_1}f(\bar{z}) \eta(z) + \pi \bar{z}Bf(z),\\
    (Bh)'(z) &= V_{h_1}h(\bar{z}) \eta(z) + \pi \bar{z}Bh(z).
\end{align*}
Thus, we get that
\begin{equation*}
    \begin{split}
        [(Bf)'Bh- Bf (Bh)'](z)  = & \left[ V_{h_1}f(\bar{z}) \eta(z) + \pi \bar{z}Bf(z) \right] Bh(z) \\
        & - Bf(z) \left[ V_{h_1}h(\bar{z}) \eta(z) + \pi \bar{z}Bh(z) \right]
    \end{split}
\end{equation*}
Since the two terms containing $\bar{z}$ as a factor cancel out we obtain that 
\begin{equation}\label{eq:estadmterm}
    [(Bf)'Bh- Bf (Bh)'](z)
    = \eta(z)\left[ V_{h_1}f(\bar{z}) Bh(z) - Bf(z) V_{h_1}h(\bar{z})\right].
\end{equation}
For the first term on the right hand side we get that 
\begin{equation*}
    \begin{aligned}
    \|\eta(z) V_{h_1}f(x,-y)Bh(z)\|_{2\pi}^2 
        &= 2 \int_\C  e^{\pi|z|^2} |V_{h_1}f(\bar{z})|^2 |Bh(z)|^2 e^{-2\pi|z|^2} \,dA(z)\\
        &\le 2 \|V_{h_1}f\|_{L^2(\C)}^2 \|(Bh) e^{-\pi|\cdot|^2/2}\|_{L^\infty(\C)}^2\\
        &\le 2 \|f\|_{L^2(\R)}^2 \|Bh\|_{\pi}^2\\
        &= 2\|f\|_{L^2(\R)}^2 \|h\|_{L^2(\R)}^2\\
        &<\infty,
    \end{aligned}
\end{equation*}
where we used the pointwise estimate given in equation \eqref{eq:fockptwise}, as well as the unitary of $B$, and the equality $\| V_{h_1}f \|_{L^2(\R^2)} = \| f \|_{L^2(\R)}$. Since the second term on the right hand side of \eqref{eq:estadmterm} can be dealt with in the same manner, we arrive at the desired conclusion.
\end{proof}
Lemma \ref{lem:admfock} enables us to apply Theorem \ref{thm:seipsampling} and Theorem \ref{thm:perelomovuniq} and we get the following result as a direct consequence of Proposition \ref{proposition:abstract_uniqueness_result}.
\begin{theorem}\label{theorem:lattice_result}
Suppose that $\mathcal{P}\subseteq\C^2$ does phase retrieval.
Furthermore, let $\Lambda\subseteq \R^2$ be separated with $D^-(\Lambda)>4.$
Then it holds that $(\{g_p\}_{p\in\mathcal{P}}, \Lambda)$ does phase retrieval on $L^2(\R).$
Moreover, if $\Lambda$ is a shifted lattice it suffices that $D^- (\Lambda)\ge 4$ for the assertion to hold. 
\end{theorem}

\subsection{Phaseless sampling in $L^2(\R,\R)$}

Consider the inclusion of Fock spaces $\mathcal{F}_\alpha^2(\C) \subseteq \mathcal{F}_\beta^2(\C)$ for $0<\alpha < \beta$ together with the characterization of sets of uniqueness among lattices and sets of stable sampling provided in Theorem \ref{thm:perelomovuniq} and Theorem \ref{thm:seipsampling}, respectively. The inclusion and the characterization formally verify the intuition that if the sampling space $\mathcal{F}_\alpha^2(\C)$ grows so does the information on the samples $Z \subseteq \C$ which are needed to recover any function in $\mathcal{F}_\alpha^2(\C)$ uniquely via sampling on $Z$. The information provided by the samples is directly related to their lower Beurling density $D^-(Z)$. This serves as the initial impetus for investigating phaseless sampling problems for function classes $\mathcal{C}$ which are proper subspaces of $\lt$. Intuitively, we expect that a restriction of the STFT phase retrieval problem to a subspace $\mathcal{C} \subseteq \lt$ leads to phaseless sampling results with reduced density. We pay particular attention to the important class $\mathcal{C} = L^2(\R,\R)$ of real-valued functions, a natural assumption made in various applications. Furthermore, the class $\mathcal{C} = L^2(\R,\R)$ is of particular interest, since in a similar fashion as for $\mathcal{C} = L^2(\R)$, functions in $L^2(\R,\R)$ are in general not determined by phaseless samples located on a lattice when only a single window function is used \cite[Section 3.3]{grohsLiehrJFAA}. 

The objective of the present section is to show that the density which is sufficient to recover real-valued functions from phaseless samples can be reduced to one-half times the sufficient density in the complex regime. We start with a Lemma on the Bargmann transform of a real-valued map.

\begin{lemma}\label{lemma:bargmann_real}
Suppose that $f \in L^2(\R,\R)$ is real-valued. Then for every $z \in \C$ and every $n \in \N_0$ it holds that
$$
\overline{(Bf)^{(n)}(z)} = (Bf)^{(n)}(\overline{z}).
$$
\end{lemma}
\begin{proof}
It follows from the definition of the Bargmann transform that if $f \in L^2(\R,\R)$ then $Bf$ is real-valued on the real axis. As a consequence, also $(Bf)^{(n)}$ is real-valued on the real axis for every $n \in \N_0$. Hence, for every $z \in \R$ and every $n \in \N_0$ we have
$
(Bf)^{(n)}(z) = \overline{(Bf)^{(n)}(\overline{z})}.
$
Observe that the right-hand side of the previous identity extends from $z \in \R$ to an entire function. The statement is therefore a consequence of the identity theorem.
\end{proof}

As an application of Lemma \ref{lemma:bargmann_real}, the next result demonstrates that a suitable decomposition of a uniqueness set for $\mathcal{A}(L^2(\R))$ leads to a reduction of uniqueness sets for $\mathcal{A}(L^2(\R,\R))$.

\begin{lemma}\label{lemma:reflection}
Let $\Lambda \subseteq \C$ and suppose that there exists a subset $\Gamma \subseteq \Lambda$  such that
$$
\Lambda = \Gamma \cup \overline{\Gamma}.
$$
If $\Lambda$ is a set of uniqueness for the class $\mathcal{A}(\lt)$ then $\Gamma$ is a set of uniqueness for the class $\mathcal{A}(L^2(\R,\R))$.
\end{lemma}
\begin{proof}
Let $f,h \in L^2(\R,\R)$ be two real-valued functions. Further, let $F = Bf$ and $H=Bh$ be the Bargmann transform of $f$ and $h$, respectively. In addition, define $G = FH(FH'-F'H)$. We have to show that if $F|_\Gamma = 0$ then $F \equiv 0$ and if $G|_\Gamma =0$ then $G \equiv 0$. We know from Lemma \ref{lemma:bargmann_real} that $\overline{F(z)} = F(\overline{z})$ and $\overline{G(z)} = G(\overline{z})$ for every $z \in \C$. Consequently, if $F$ vanishes at $\gamma \in \Gamma$ then $F$ vanishes at $\overline{\gamma} \in \overline{\Gamma}$ and the same property holds for the function $G$. This shows that both $F$ and $G$ vanish at $\Gamma \cup \overline{\Gamma} = \Lambda$. But $\Lambda$ is a uniqueness set for the class $\mathcal{A}(\lt) \supseteq \mathcal{A}(L^2(\R,\R))$. Therefore, $F=0$ and $G=0$ which yields the assertion.
\end{proof}

In a similar fashion as in Section \ref{subsection:phaseless_sampling}, we obtain the following abstract uniqueness result for the function class $\mathcal{C}=L^2(\R,\R)$ which follows directly from Lemma \ref{lemma:reflection} and Proposition \ref{proposition:abstract_uniqueness_result}.

\begin{proposition}\label{proposition:abstract_uniqueness_result_real_case}
Suppose that $\mathcal{P}\subseteq\C^2$ does phase retrieval.
Moreover, suppose that $\Lambda \subseteq \C$ has the property that $\overline{\Lambda}$ is a set of uniqueness for the class $\mathcal{A}(\lt)$. Further, assume that $\Gamma \subseteq \Lambda$ satisfies $\Lambda = \Gamma \cup \overline{\Gamma}$. Then it holds that $(\{g_p\}_{p\in\mathcal{P}}, \Gamma)$ does phase retrieval on $L^2(\R,\R)$.
\end{proposition}

For instance, in the situation of Proposition \ref{proposition:abstract_uniqueness_result_real_case}, we observe that if $\Lambda$ is a set of uniqueness for the class $\mathcal{A}(\lt)$ which is symmetric with respect to the real axis, i.e. $\Lambda = \overline{\Lambda}$, then $(\{g_p\}_{p\in\mathcal{P}}, \Lambda \cap \mathrm{cl}(\mathbb{H}^+))$ does phase retrieval on $\mathcal{A}(L^2(\R,\R))$ where $\mathbb{H}^+ \coloneqq \{ z \in \C : \im(z) > 0 \}$ denotes the open upper half-plane, and $\mathrm{cl}(\mathbb{H}^+)$ denotes the closure of $\mathbb{H}^+$. In fact, phase retrieval on $\mathcal{A}(L^2(\R,\R))$ is possible from shifted lattices $\Lambda \subseteq \R^2$ with density $D^-(\Lambda) \geq 2$. Thus, in comparison to the complex case, a density reduction by one-half is achievable.

\begin{theorem}\label{theorem:separable_lattice_uniqueness_real_valued}
Suppose that $\mathcal{P}\subseteq\C^2$ does phase retrieval.
Furthermore, let $\Lambda =\alpha \Z \times \beta \Z$ with $\alpha,\beta \in \R \setminus \{ 0 \}$ be a separable lattice of density $D^-(\Lambda) \geq 2$.
Then it holds that $(\{g_p\}_{p\in\mathcal{P}}, (0,\tfrac{\beta}{4})^T+\Lambda)$ does phase retrieval on $L^2(\R,\R)$.
\end{theorem}

The proof of the above statement is based on the following Lemma.

\begin{lemma}\label{lemma:decomposition_of_lattices}
Suppose that $\Lambda \subseteq \R^2$ is a shifted lattice of the form $\Lambda=(0,\frac{\beta}{2})^T + (\alpha \Z + \beta \Z)$ with $\alpha,\beta \in \R \setminus \{ 0 \}$. Further, define shifted lattices $\Gamma_1, \Gamma_2 \subseteq \R^2$ by 
$$
\Gamma_1 = \begin{pmatrix} 0 \\ \tfrac{\beta}{2} \end{pmatrix} + \begin{pmatrix} \alpha & 0 \\ 0 & 2\beta \end{pmatrix} \Z^2, \ \ \ \Gamma_2 = \begin{pmatrix} 0 \\ \tfrac{\beta}{2} \end{pmatrix} + \begin{pmatrix} \alpha & 0 \\ \beta & 2\beta \end{pmatrix} \Z^2
$$
Then the following holds:
\begin{enumerate}[i)]
    \item $\Gamma_1$ and $\Gamma_2$ are shifted sub-lattices of $\Lambda$
    \item It holds that $D^-(\Lambda) = (\alpha \beta)^{-1}$ and that $D^{-}(\Gamma_1) = D^{-}(\Gamma_2) = (2 \alpha \beta)^{-1}$
    \item $\Lambda = \Gamma_1 \cup \overline{\Gamma_1}= \Gamma_2 \cup \overline{\Gamma_2}$.
\end{enumerate}
\end{lemma}
\begin{proof}
The fact that $\Gamma_1$ and $\Gamma_2$ are shifted sub-lattices of $\Lambda$ follows directly from the definition of $\Gamma_1$ and $\Gamma_2$.

The second statement follows directly from the fact that for every shifted lattice $v+A\Z^2$, $A \in \mathrm{GL}_2(\R)$, $v\in \R^2,$ it holds that $D^{-}(v+A\Z^2) = |\det(A)|^{-1}$.

It remains to prove the third claim. We start by showing the inclusion $\Lambda \subseteq  \Gamma_1 \cup \overline{\Gamma_1}$. 
To that end, pick an arbitrary $z \in \Lambda$. Then there exists $(n,k)^T \in \Z^2$ such that
\begin{equation}\label{eq:point_z}
    z = \begin{pmatrix} \alpha n \\ \beta(k+\frac{1}{2}) \end{pmatrix}.
\end{equation}
If $k$ is even then $k=2m$ for some $m \in \Z$ and this gives
$$
z = \begin{pmatrix} 0 \\ \tfrac{\beta}{2} \end{pmatrix} + \begin{pmatrix} \alpha & 0 \\ 0 & 2\beta \end{pmatrix} \begin{pmatrix} n \\ m \end{pmatrix} \in \Gamma_1.
$$
On the other hand, if $k$ is odd then $k=2m-1$ for some $m \in \Z$ which yields
\begin{equation*}
    \begin{split}
        z & = \begin{pmatrix} \alpha n \\ \tfrac{\beta}{2} + \beta(2m-1) \end{pmatrix} = \begin{pmatrix} \alpha n \\ -\tfrac{\beta}{2}  -2\beta(-m) \end{pmatrix} \\
        & = \overline{\begin{pmatrix} 0 \\ \tfrac{\beta}{2} \end{pmatrix} + \begin{pmatrix} \alpha & 0 \\ 0 & 2\beta \end{pmatrix}  \begin{pmatrix} n \\ -m \end{pmatrix}} \in \overline{\Gamma_1}.
    \end{split}
\end{equation*}
As a result, the inclusion $\Lambda \subseteq \Gamma_1 \cup \overline{\Gamma_1}$ holds true and we are left with the proof that $\overline{\Gamma_1} \subseteq \Lambda$. This, however, is a consequence of the fact that for every $(n,k)^T \in \Z^2$ we have
\begin{equation*}
    \begin{split}
        & \overline{\begin{pmatrix} 0 \\ \tfrac{\beta}{2} \end{pmatrix} + \begin{pmatrix} \alpha & 0 \\ 0 & 2\beta \end{pmatrix}  \begin{pmatrix} n \\ k\end{pmatrix}} \\ &= \begin{pmatrix} \alpha n \\ -\tfrac{\beta}{2} - 2\beta k \end{pmatrix} 
        = \begin{pmatrix} 0 \\ \tfrac{\beta}{2} \end{pmatrix} + \begin{pmatrix} \alpha & 0 \\ 0 & \beta \end{pmatrix}  \begin{pmatrix} n \\ 1-2(k+1)\end{pmatrix} \in \Lambda.
    \end{split}
\end{equation*}
In the remainder of the proof we establish the equality $\Lambda = \Gamma_2 \cup \overline{\Gamma_2}$. In a similar fashion as above we pick a point $z \in \Lambda$ and observe that $z$ is given as in equation \eqref{eq:point_z} for some $(n,k)^T \in \Z^2$. 
If either both $n,k$ are even or both $n,k$ are odd then $k-n$ is an even integer and there exists an $m \in \Z$ such that $k-n=2m$, or equivalently,
$$
k+\tfrac{1}{2} = n + 2m + \tfrac{1}{2}.
$$
This shows that the point $z$ as defined in \eqref{eq:point_z} satisfies the relation
$$
z =  \begin{pmatrix} \alpha n \\ \beta(n + 2m + \tfrac{1}{2})\end{pmatrix} =  \begin{pmatrix} 0 \\ \tfrac{\beta}{2} \end{pmatrix} + \begin{pmatrix} \alpha & 0 \\ \beta & 2 \beta \end{pmatrix} \begin{pmatrix} n \\ m \end{pmatrix} \in \Gamma_2.
$$
On the other hand, if $n,k$ are such that one of them is even and one of them is odd then $k+n$ is an odd number. Thus, there exists an $m \in \Z$ such that $k+n+1 = -2m$, or equivalently,
$$
k + \tfrac{1}{2} = - (n+2m+\tfrac{1}{2}),
$$
which yields
$$
z = \begin{pmatrix} \alpha n \\ -\beta(n + 2m + \tfrac{1}{2})\end{pmatrix} =  \overline{\begin{pmatrix} 0 \\ \tfrac{\beta}{2} \end{pmatrix} + \begin{pmatrix} \alpha & 0 \\ \beta & 2 \beta \end{pmatrix} \begin{pmatrix} n \\ m \end{pmatrix}} \in \overline{\Gamma_2}.
$$
Hence, we have proved that $\Lambda \subseteq \Gamma_2 \cup \overline{\Gamma_2}$. To prove the desired equality, it remains to show that $\overline{\Gamma_2} \subseteq \Lambda$. But this follows readily by picking an arbitrary $(n,k)^T \in \Z^2$ and observing that
\begin{equation*}
    \begin{split}
        \overline{\begin{pmatrix} 0 \\ \tfrac{\beta}{2} \end{pmatrix} + \begin{pmatrix} \alpha & 0 \\ \beta & 2\beta \end{pmatrix} \begin{pmatrix} n \\ k \end{pmatrix}} & =\begin{pmatrix} \alpha n \\ -\tfrac{\beta}{2} - \beta n - 2\beta k \end{pmatrix} \\
        &= \begin{pmatrix} 0 \\ \tfrac{\beta}{2} \end{pmatrix} + \begin{pmatrix} \alpha & 0 \\ 0 & \beta \end{pmatrix}  \begin{pmatrix} n \\ -n-2k-1 \end{pmatrix} \in \Lambda.
    \end{split}
\end{equation*}
\end{proof}

\begin{proof}[Proof of Theorem \ref{theorem:separable_lattice_uniqueness_real_valued}]
Notice that the shifted lattice $\Gamma \coloneqq (0,\tfrac{\beta}{4})^T + \Lambda$ as given in the statement of the theorem is equal to the lattice $\Gamma_1$ as given in Lemma \ref{lemma:decomposition_of_lattices} with $\beta$ replaced by $\frac{\beta}{2}$. Additionally, Lemma \ref{lemma:decomposition_of_lattices}(iii) shows that
$$
\Gamma \cup \overline{\Gamma} = (0,\tfrac{\beta}{4})^T + (\alpha \Z \times \tfrac{\beta}{2}\Z).
$$
Since by assumption we have $D^-(\Lambda) = (\alpha \beta)^{-1} \geq 2$ it follows that the shifted lattice $\Gamma \cup \overline{\Gamma}$ satisfies $D^-(\Gamma \cup \overline{\Gamma}) = 2(\alpha \beta)^{-1} \geq 4$. In particular, $\Gamma \cup \overline{\Gamma}$ is a set of uniqueness for $\mathcal{A}(\lt)$. Proposition \ref{proposition:abstract_uniqueness_result_real_case} yields the assertion.
\end{proof}

Notice that Lemma \ref{lemma:decomposition_of_lattices} yields two possibilities of shifted lattices which guarantee phase retrieval in $L^2(\R,\R)$, namely $\Gamma_1$ and $\Gamma_2$. The shifted lattice $\Gamma_1$ is separable and $\Gamma_2$ arises from a lattice of the form $a\Z^2, a>0$, via a rotation followed by a translation. Figure \ref{figure:lattice} depicts the lattice $\Gamma_2$.

\begin{figure}[h]
\centering
\hspace*{-0.3cm}
  \includegraphics[width=13cm]{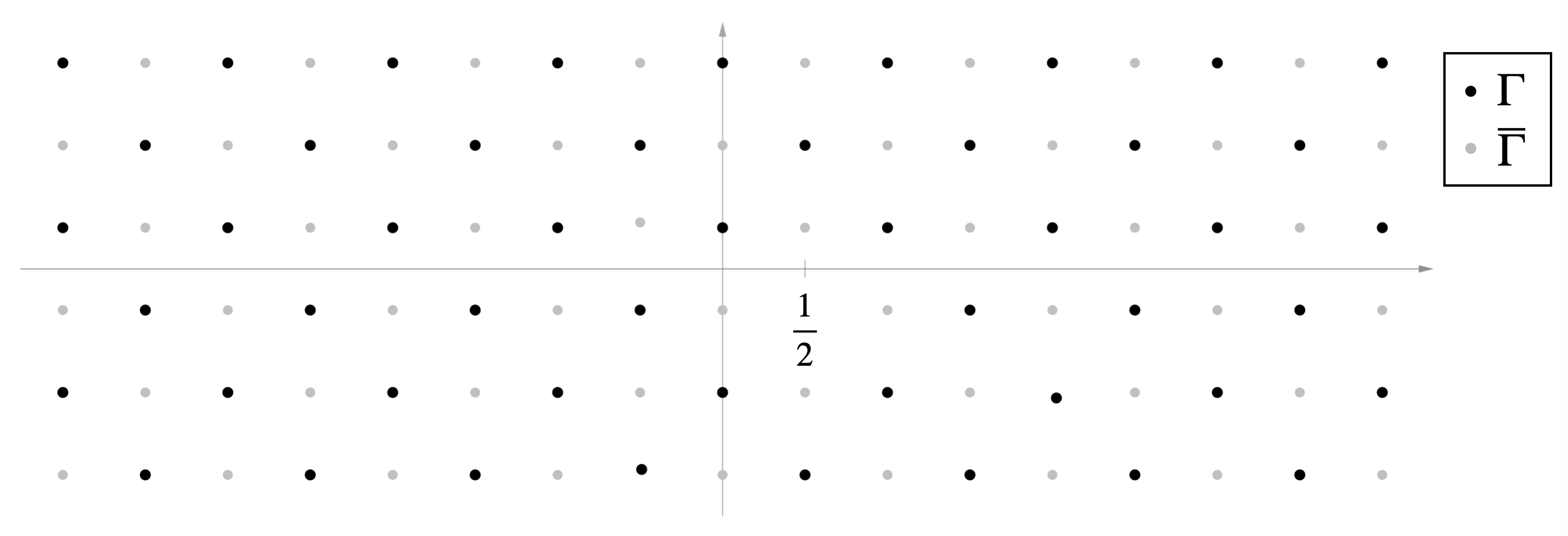}
\caption{A shifted lattice $\Gamma$ and its complex conjugate $\overline{\Gamma}$. The union $\Gamma \cup \overline{\Gamma}$ is the shifted rectangular lattice $\Lambda = (0,1/4)+\frac{1}{2}\Z^2$ with $D^-(\Lambda)=4$. In the situation of Theorem \ref{theorem:lattice_result}, $(\{g_p\}, \Lambda)$ does phase retrieval on $L^2(\R)$ whereas $(\{g_p\}, \Gamma)$ does phase retrieval on $L^2(\R,\R)$.}
\label{figure:lattice}
\end{figure}

\section*{Acknowledgements}
Martin Rathmair was supported by the Erwin–Schr{\"o}dinger Program (J-4523) of the Austrian Science Fund (FWF).

\bibliographystyle{abbrv}
\bibliography{bibfile}

\end{document}